\documentclass[11pt,final,reqno]{article}
\usepackage{fullpage}
\usepackage{amsmath,amsfonts,amssymb,amsthm}
\usepackage{graphicx}
\usepackage{epstopdf}
\usepackage{color}
\usepackage{epsf}
\usepackage{verbatim}
\usepackage{enumitem}
\usepackage{hyperref}
\usepackage{subfig} 
\usepackage{IEEEtrantools}

\newtheorem{theorem}{Theorem}[section]
\newtheorem{lemma}[theorem]{Lemma}
\newtheorem{proposition}[theorem]{Proposition}

\newtheorem{example}[theorem]{Example}
\newtheorem{problem}{Problem}
\newtheorem{definition}[theorem]{Definition}

\newcommand{\bF}{\textbf{F}}

\newcommand{\cF}{{\mathcal F}}
\newcommand{\cG}{{\mathcal G}}
\newcommand{\cH}{{\mathcal H}}

\newcommand{\cR}{{\mathcal R}}

\newcommand{\del}{\delta}

\newcommand{\ep}{\varepsilon}

\newcommand{\floor}[1]{\left\lfloor#1\right\rfloor}
\newcommand{\ceiling}[1]{\left\lceil#1\right\rceil}

\newcommand{\ex}{\mathrm{ex}}

\newcommand{\urltilde}{\kern -.15em\lower .7ex\hbox{~}\kern .04em}

\title{On the co-degree threshold for the Fano plane}

\author{
Louis DeBiasio\thanks{Dept. of Mathematics, Miami University, Oxford, OH 45056, USA. E-mail: debiasld@muohio.edu.}
 \quad and \quad
Tao Jiang\thanks{Dept. of Mathematics, Miami University, Oxford,
OH 45056, USA. E-mail: jiangt@muohio.edu. }  }
\date{\today}
\begin{document}
\maketitle

\begin{abstract}
Given a $3$-graph $H$, 
let $\ex_2(n, H)$ denote the maximum value of the minimum co-degree of a $3$-graph on $n$ vertices which does not contain a copy of $H$.
Let $\bF$ denote the Fano plane, which is  the $3$-graph $\{axx',ayy',azz',xyz',xy'z,x'yz,x'y'z'\}$. 
Mubayi \cite{M} proved that $\ex_2(n,\bF)=(1/2+o(1))n$ and conjectured that $\ex_2(n, \bF)=\floor{n/2}$ for sufficiently large $n$.  Using a very sophisticated quasi-randomness argument, Keevash \cite{K} proved Mubayi's conjecture.  Here we give a simple proof of Mubayi's conjecture by using a class of $3$-graphs that we call rings.  We also determine the Tur\'an density of the family of rings.
\end{abstract}

\section{Introduction}

For a family $\mathcal{H}$ of $k$-graphs, let $\ex(n, \mathcal{H})$ denote the maximum number of edges in an $n$-vertex $k$-graph which contains no member of $\mathcal{H}$.  Determining $\ex(n, \mathcal{H})$ is a fundamental question in graph theory which becomes extremely difficult when $k\geq 3$.  Let $\pi(\mathcal{H})=\lim_{n\to \infty}\frac{\ex(n, \mathcal{H})}{\binom{n}{k}}$ and call this value the \emph{Tur\'an density} of $\mathcal{H}$ (as has been pointed out many times, it is easy to show that this limit exists). 
When $\mathcal{H}$ consists of a single graph $H$, we write $\pi(H)$ for $\pi(\mathcal{H})$. Let $K_4^3$ denote the complete $3$-graph on four vertices. Over 70 years ago, Tur\'an famously conjectured that $\pi(K_4^3)=\frac{5}{9}$, but this conjecture is still unproved \cite{T}.  In fact, when $k\geq 3$ there are very few $k$-graphs for which the Tur\'an density is known (see \cite{K2} for a detailed account).  Despite this general difficulty, there is a special $3$-graph called the \emph{Fano plane} for which much is known.

The Fano plane, denoted $\textbf{F}$, is the projective geometry of dimension $2$ over the field with $2$ elements; alternatively, $\textbf{F}$ is the $3$-graph on seven vertices $\{a,x,y,z,x',y',z'\}$ with the seven edges $\{axx',ayy',azz',xyz',xy'z,x'yz,x'y'z'\}$.  
Let $B(n)$ denote the balanced complete bipartite $3$-graph, which is obtained by partitioning
a set of $n$ vertices into parts of size $\ceiling{n/2}$ and $\floor{n/2}$ and taking
as edges all the triples intersecting both parts. 
Since $B(n)$ is $2$-colorable and it is easy to see that $\bF$ is not, $B(n)$ contains no copy of $\bF$. Therefore, $\ex(n,\bF)\geq e(B(n))= 
\binom{n}{3}-\binom{\floor{n/2}}{3}-\binom{\ceiling{n/2}}{3}$.
S\'os \cite{Sos} conjectured that this lower bound is asymptotically best possible and hence $\pi(\textbf{F})=\frac{3}{4}$.  
A few decades later, de Caen and F\"uredi \cite{DF} proved S\'os' conjecture via a clever use of  so-called link graphs.  A few years later, Keevash and Sudakov \cite{KS} and independently F\"uredi and Simonovits \cite{FS} proved the exact counterpart of this result; that is, $\ex(n, \textbf{F})=\binom{n}{3}-\binom{\floor{n/2}}{3}-\binom{\ceiling{n/2}}{3}$ for sufficiently large $n$.

Let $G$ be a $k$-graph with vertex set $V$.   Given any subset $U\subseteq V$, $|U|\leq k$, 
the {\it degree} of $U$, denoted by $d(U)$, is the number of edges of $G$ that contain $U$.  For simplicity, when $U$ consists
of one vertex $x$  or two vertices $x$ and $y$, we write
 $d(x)$, and $d(x,y)$ instead of $d(\{x\})$ and $d(\{x,y\})$, respectively.  
When $k=3$, we call $d(x,y)$ the {\it co-degree} of $x$ and $y$, while the 
set of vertices $z$ such that $xyz\in E(G)$ is called the {\it co-neighborhood} of $x,y$ and
will be denoted by $N(x,y)$.
For each integer $0\leq \ell \leq k$, let  $\del_{\ell}(G) = \min\{d(U): U\subseteq V, |U|=\ell\}$. We call
$\del_{\ell}(G)$ the minimum {\it $\ell$-degree} of $G$.  For a family $\mathcal{H}$ of $k$-graphs, let $\ex_\ell(n, \mathcal{H})$ denote the maximum value of $\delta_\ell(G)$ in an $n$ vertex $k$-graph $G$ which contains no member of $\mathcal{H}$ and let $\pi_\ell(\mathcal{H})=\lim_{n\to \infty}\frac{\ex_\ell(n,\mathcal{H})}{\binom{n-\ell}{k-\ell}}$.  Mubayi and Zhao \cite{MZ} prove that this limit exists in the case $\ell=k-1$ and Lo and Markstr\"om \cite{LM} prove that this limit exists for all $0\leq \ell\leq k-1$ (a fact previously sketched by Keevash \cite{K}). Note that the case $\ell=0$ just reduces to $\pi(\mathcal{H})$. 
When $k=3$, we call $\pi_2(\mathcal{H})$ the {\it co-degree density} of
$\mathcal{H}$. For general $k$-graphs, a simple averaging argument shows that $\pi_i(\mathcal{H})\geq \pi_j(\mathcal{H})$ when $i\leq j$
(see \cite{K2} Section 13.2). It is also pointed out in \cite{K2} Section 13.2 that for any graph $H$,
$\pi_1(H)=\pi(H)$. The same argument applies to any finite family $\cH$ as well.

\begin{proposition}  \label{min-degree}
For a finite family $\mathcal{H}$ of $k$-graphs, $\pi_1(\mathcal{H})=\pi(\mathcal{H})$.
\end{proposition}
\begin{proof} (sketch).  Let $a=\pi_1(\cH)$.
Let $\ep>0$ be any small positive real. Let $n$ be sufficiently large as a function of $\ep$.
Let $G$ be a $k$-graph with $e(G)>(a+\ep)\binom{n}{k}$. By \cite{K2} Proposition 4.2, $G$ contains
a subgraph $G'$ on $m=\Omega(n)$ vertices with $\delta_1(G')\geq (a+\frac{\ep}{2})\binom{m-1}{k-1}$.
Since $\pi_1(\cH)=a$ and $m\to \infty$ as $n\to \infty$, when $n$ is large enough, we have $\delta_1(G')>\ex_1(\cH)$.
So $G'$ contains a member of $\cH$ and therefore $G$ contains a member of $\cH$.
\end{proof}

So the minimum degree problem is essentially the same as the Tur\'an problem. The minimum co-degree problem
however is drastically different. For instance, there are $3$-graphs $H$ with $\pi(H)$ arbitrarily close to $1$ and yet $\pi_2(H)=0$ (see \cite{MZ}).
In general, there has not been a very good understanding  of the relationship between $\pi(H)$ and $\pi_2(H)$
(see \cite{MZ} and \cite{LM} for detailed discussions).
Similar to the situation with the Tur\'an density, not much is known about $\pi_2(H)$ even for small graphs $H$ such as $K_4^3$ (in this case Czygrinow and Nagle \cite{CN} conjectured that $\pi_2(K_4^3)=\frac{1}{2}$).  
Mubayi \cite{M} initiated the study of $\ex_2(n,\bF)$, where $\bF$ is the Fano plane.
As pointed out earlier, $B(n)$ contains no copy of $\bF$.
So,  $\ex_2(n,\bF)\geq \delta_2(B(n))=\floor{n/2}$. 

Mubayi \cite{M} proved an asymptotically matching upper bound thus establishing $\pi_2(\textbf{F})=\frac{1}{2}$.  
He further conjectured that $\ex_2(n,\textbf{F})=\floor{\frac{n}{2}}$, for sufficiently large $n$. 
This was later proved by Keevash \cite{K} using a very sophisticated argument involving hypergraph regularity, quasi-randomness, and stability (We should mention that Keevash proves the stronger statement that the extremal example is ``stable".  Also, the scope of Keevash's paper is not limited to the problem of determining the co-degree threshold for the Fano plane.).
In this paper, we give a simple proof of Mubayi's conjecture which is in the same spirit as Mubayi's original proof of $\pi_2(\textbf{F})=\frac{1}{2}$. Our main result is

\begin{theorem}\label{main}
There exists $n_0$ such that if $n\geq n_0$, then $\ex_2(n, \bF)=\floor{\frac{n}{2}}$.
\end{theorem}

Since we are giving a new proof of an old result, it is worth mentioning that we only need $n_0$ to be large enough so that ``supersaturation" holds (see Section \ref{lemmas}).  While we do not make an attempt to compute the value of $n_0$, it is considerably smaller than the value of $n_0$ needed for the use of regularity in \cite{K}.

The paper is organized as follows.  In Section \ref{lemmas} we give some Lemmas and introduce a family of $3$-graphs called
{\it rings}.  In Section \ref{fano} we prove Theorem \ref{main} by making use of the family of rings. In Section \ref{turan} we determine the Tur\'an density of the family of rings.  
 Finally, in Section \ref{remarks}, we conclude with some remarks and open problems.

\section{Lemmas}\label{lemmas}

For any $k$-graph $G$, the \emph{$s$-blowup} of $G$, denoted $G(s)$, is the graph obtained from $G$ by cloning each vertex $s$ times.  For a family of $k$-graphs $\mathcal{H}$, let $\mathcal{H}(s)=\{H(s): H\in \mathcal{H}\}$.   Erd\H{o}s \cite{E} used supersaturation to show 

\begin{lemma}  \label{erdos-blowup}\cite{E}
For any finite family of $k$-graphs $\cH$ and any positive integer $s$,
$\pi(\mathcal{H})=\pi(\mathcal{H}(s))$.
\end{lemma}

 Keevash and Zhao \cite{KZ} proved an analogous result for the co-degree density.

\begin{lemma}\label{codegreeblowup}\label{KZ}
For any finite family of $k$-graphs $\cH$ and any positive integer $s$,
$\pi_2(\mathcal{H})=\pi_2(\mathcal{H}(s))$.
\end{lemma}

The same supersaturation argument in fact gives
\begin{lemma}\label{blowup}
For any finite family of $k$-graphs $\cH$ and any positive integer $s$, and any $j$, $0\leq j\leq k-1$,
$\pi_j(\mathcal{H})=\pi_j(\mathcal{H}(s))$.
\end{lemma}

We also make the following trivial observation based on the definitions.

\begin{proposition} \label{subgraph}
Let $\cH$ and $\cG$ be two families of $k$-graphs. Let $j\in\{0,\ldots, k-1\}$.  Suppose that for every member $G\in \cG$,
some subgraph of $G$ belongs to $\cH$. Then $\ex_j(n,\cH)\leq \ex_j(n,\cG)$. So, in particular,
$\pi_j(\cH)\leq \pi_j(\cG)$. 
\end{proposition}

We now define a family of $3$-graphs, called \emph{rings}, which will play a central role in our proof of Theorem \ref{main}.

\begin{definition}
Let $t\geq 2$ and let $V$ be a set of at most $2t$ vertices surjectively labeled with $x_0, y_0, x_1, y_1, \dots, x_{t-1}, y_{t-1}$.  Let $\cR^*_t$ be the family of $3$-graphs on $V$ with edge set $\bigcup_{i=0}^{t-1}\{x_i, y_i, x_{i+1}\}\cup \{x_i, y_i, y_{i+1}\}$, where addition is defined modulo $t$.  Let $R_t$ be the (unique) member of $\cR^*_t$ which has exactly $2t$ vertices and call $R_t$ a ring on $2t$ vertices.  Let $\cR^*_{\leq t}=\bigcup_{i=2}^t \cR_i^*$  and $\cR_{\leq t}=\{R_2, R_3, \dots, R_t\}$.
\end{definition}

\begin{lemma}\label{familyblowup}
For all positive integers $t\geq 2$ and $0\leq j\leq2$, we have
$\pi_j(\cR_{\leq t})=\pi_j(\cR^*_{\leq t})$ and $\pi_j(R_t)=\pi_j(\cR_t^*)$.
\end{lemma}
\begin{proof}
Since $\cR_{\leq t}\subseteq \cR^*_{\leq t}$, we have $\pi_j(\cR_{\leq t}^*)\leq \pi_j(\cR_{\leq t})$.
On the other hand, for every $i\leq t$,  $\cR_i^*(t)$ clearly contains a copy of $R_i$, since in any member of $\cR_i^*(t)$ 
there are $t$ distinct copies of $x_i,y_i$. By Proposition \ref{subgraph} and Lemma \ref{codegreeblowup},
$\pi_j(\cR_{\leq t})\leq \pi_j(\cR_{\leq t}^*(t))=\pi_j(\cR_{\leq t}^*)$.
Thus, $\pi_j(\cR_{\leq t})=\pi_j(\cR_{\leq t}^*)$. 

By a similar argument, we have $\pi_j(R_t)=\pi_j(\cR^*_t)$.
\end{proof}

\begin{definition}\label{lmproperty}
A hypergraph $H$ on $l$ vertices is said to have the {\it $(l,m)$-property} if every subset of $m$ vertices contains at least one edge of $H$.
\end{definition}

 Mubayi and R\"odl \cite {MR} recursively constructed for every $t\geq 2$
a family $\mathcal{F}_t$ of $3$-graphs with the $(2t+1,t+2)$-property. They showed that 
$\pi(\mathcal{F}_t)\leq \frac{1}{2}$ for each fixed $t\geq 2$ and used this to establish an upper bound on the Tur\'an density of $\{abc, ade, bde, cde\}$ (sometimes referred to as the $3$-book with $3$ pages). This family $\mathcal{F}_t$ also played a key role in Mubayi's proof of $\pi_2(\bF)=\frac{1}{2}$.
Here, we observe that for every $t$ the graph $R_t$ has the $(2t,t+1)$-property and we will also show
that $\pi_2(\cR_{\leq t})$ is small.  Then, by using $\cR_{\leq t}$ instead of $\cF_t$ 
we are able to establish $\ex_2(n,\bF)=\floor{\frac{n}{2}}$.

\begin{lemma}\label{edge}
$R_t$ has the $(2t,t+1)$-property. 
\end{lemma}

\begin{proof}
Clearly $R_t$ has $2t$ vertices.
Let $S$ be any set of vertices in $R_t$ that contains no edge. We show that $|S|\leq t$.  For each $i\in I=\{0,1,\dots, t-1\}$, if $x_i, y_i\in S$ then $x_{i+1},y_{i+1}\notin S$ (addition modulo $t$) otherwise we would have an edge.  This implies $|S|\leq t$. 
\end{proof}

Next, we show that $\pi_2(\cR_{\leq t})$ is small by using an auxiliary directed graph. First we recall some old results concerning
short directed cycles in directed graphs. As usual,
for a directed graph $D$, let $\delta^+(D)$ and $\delta^-(D)$ denote the minimum out-degree and in-degree of $D$ respectively.  Caccetta and H\"aggkvist \cite{CH} conjectured that if $D$ is a directed graph on $n$ vertices with $\delta^+(D)\geq r$, then $D$ contains a cycle of length at most $\ceiling{\frac{n}{r}}$.  While their conjecture remains open, Chv\'atal and Szemer\'edi \cite{CS} gave a simple proof of a slightly weaker statement.

\begin{theorem}[Chv\'atal-Szemer\'edi]\label{ChSz}
Let $D$ be a directed graph on $n$ vertices.  If $\delta^+(D)\geq r$ (or $\delta^-(D)\geq r$), then $D$ contains a directed cycle of length at most $\frac{2n}{r+1}$.
\end{theorem}

There have been improvements on this result. However, Theorem \ref{ChSz} suffices for our purposes.

\begin{theorem} \label{ring-codegree}
For all $t\geq 2$ we have $\pi_2(\cR_{\leq t})\leq \frac{\sqrt{2}}{\sqrt{t}}$.
\end{theorem}
\begin{proof}
By Lemma \ref{familyblowup}, it suffices to prove that $\pi_2(\cR^*_{\leq t})\leq \frac{\sqrt{2}}{\sqrt{t}}$. Let $a=\frac{\sqrt{2}}{\sqrt{t}}$. Let $\epsilon$ be a small positive real and let $b=a+\epsilon$. Let $n$ be
sufficiently large as a function of $\epsilon$.
Let $G$ be a $3$-graph on $n$ vertices with $\delta_2(G)\geq \ceiling{bn}$. 

Let $D$ be an auxiliary digraph with vertex set $\binom{V(G)}{2}$ such that $(\{u,v\}, \{u',v'\})$ is an edge of $D$ if and only if $uvu'$ and $uvv'$ are edges of $G$ (in other words, if and only if $u',v'\in N_G(u,v)$).  Let $N=\binom{n}{2}$.  Then $D$ has $N$ vertices. 
 For any $\{u,v\}\in V(D)$, its out-neighbors in $D$ are precisely all the $2$-subsets of  $N_G(u,v)$ and thus  (using $n$ being sufficiently large)
\begin{equation*}
\delta^+(D)\geq \binom{\ceiling{bn}}{2}\geq\frac{bn(bn-1)}{2}\geq \frac{a^2 n^2}{2}\geq a^2\binom{n}{2}= \frac{2N}{t}.
\end{equation*}

By Theorem \ref{ChSz}, $D$ contains a directed cycle $C$ of length at most $\frac{2N}{2N/t+1}\leq t$.
The subgraph of $G$ corresponding to $C$ is a member of $\cR^*_{\leq t}$.
\end{proof}

\section{The co-degree threshold for the Fano plane} \label{fano}

Let $\bF^*$ be the $3$-graph obtained from the complete $3$-partite $3$-graph with vertex set $\{x,x',y,y',z,z'\}$ by adding the vertex $u$ and the three edges $uxx',uyy',uzz'$.  Notice that $\bF\subseteq \bF^*$.  We obtain Theorem \ref{main} as a corollary of the following more general theorem.

\begin{theorem}\label{F*}
For sufficiently large $n$, $\ex_2(n, F)=\floor{\frac{n}{2}}$ for all $\bF\subseteq F\subseteq \bF^*$.
\end{theorem}

\begin{proof}
In the introduction we pointed out that $B(n)$ gives the lower bound $\ex_2(n,\bF)\geq \floor{\frac{n}{2}}$, thus it suffices to prove $\ex_2(n, \bF^*)\leq \floor{\frac{n}{2}}$.  

By Theorem \ref{ring-codegree} and Lemma \ref{codegreeblowup}, $\pi_2(\cR_{\leq 9}(2))\leq \frac{\sqrt{2}}{3}<\frac{1}{2}$.  Let $n$ be large enough such that $\ex_2(n, \cR_{\leq 9}(2))<\floor{\frac{n}{2}}$.
Let $G$ be a graph on $n$ vertices with $\delta_2(G)\geq \floor{\frac{n}{2}}+1$.  
Then $G$ contains a copy $H$ of $R_t(2)$ for some $t\leq 9$.

  For each vertex $v$ in $R_t$, let $v'$ denote the clone of $v$ in $R_t(2)$.  For all $v\in R_t$, let $C_v=\{u\in V(G): u\in N(v,v')\}$.  Summing over all $v\in R_t$ and using the exact condition $\delta_2(G)\geq \floor{\frac{n}{2}}+1$ (the only place where the exact condition is needed), gives 
\begin{equation}\label{exact}
\sum_{v\in R_t} |C_v|\geq 2t\left(\floor{\frac{n}{2}}+1\right)\geq 2t\left(\frac{n+1}{2}\right)>tn.  
\end{equation}
This implies that there exists some $u^*\in V(G)$ which is contained in more than $t$ different sets $C_v$.  Therefore, by Lemma \ref{edge}, there are vertices $x,y,z\in R_t$ such that $xyz$ is an edge in $R_t$ with $u^*\in C_x$, $u^*\in C_y$ and $u^*\in C_z$.  So in $R_t(2)$, $S:=\{x,x',y,y',z,z'\}$ induces a complete $3$-partite $3$-graph and thus $S\cup \{u^*\}\cong F^*$ (see Figure \ref{Fplus}).

\begin{figure}[h]
\begin{center}
\scalebox{.75}{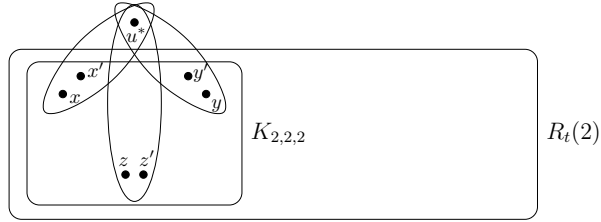}
\caption{Obtaining the subgraph $F^*$}
\label{Fplus}
\end{center}
\end{figure}

\end{proof}

\section{Tur\'an density of rings}\label{turan}

Let $\cR$ denote the family $\cup_{i\geq 2} \{R_i\}$.  
The fact that $R_t$ has the $(2t, t+1)$-property and the family $\cR_{\leq t}$ has small
co-degree density was key to our short proof of Mubayi's conjecture.
Conceivably, the family $\cR_{\leq t}$ can be useful elsewhere in the study of the Tur\'an
problem for $3$-graphs. For instance, if $\cR_{\leq t}$ also has relatively small Tur\'an density,
then it could potentially be used in bounding the Tur\'an densities of other $3$-graphs, just like
how $\cF_t$ was used by Mubayi and R\"odl \cite{MR}. In this section, we show that
similar to $\cF_t$ the family $\cR$ also has Tur\'an density at most $\frac{1}{2}$.
In fact, we will show that the Tur\'an density of $\mathcal{R}$ is exactly $\frac{1}{2}$.
The family $\cR$ does, however, have some advantages over $\cF_t$.
One,  it has the $(2t,t+1)$-property versus $\cF_t$ having the $(2t+1, t+2)$-property.
Two, the structure of $R_t$ is simple and explicit, while in forcing a member of $\cF_t$, we
do not quite know which particular structure that member has.

Next, we show that $\cR$ has Tur\'an density at least $\frac{1}{2}$ via a
construction inspired by the ``half-graph'' constructions from bandwidth problems.

\begin{example} \label{ring-lower-construction}
Let $A=\{a_1, a_2,\dots, a_{\floor{n/2}}\}$ and $B=\{b_1, b_2, \dots, b_{\ceiling{n/2}}\}$. 
Let $G_n$ be a $3$-graph on $A\cup B$ whose edges are all the triples of the form $\{a_i, b_j, a_k\}$ and 
$\{a_i, b_j, b_k\}$ where $i,j<k$.   
\end{example}

It is easy to check that $\lim_{n\to \infty} e(G_n)/\binom{n}{3}=\frac{1}{2}$.

\begin{proposition}\label{ring-lower}
For all $n$ the graph $G_n$ given in Example \ref{ring-lower-construction} contains no member of $\mathcal{R}$
and hence $\pi(\mathcal{R})\geq \frac{1}{2}$.
\end{proposition}
\begin{proof}
Observe first that, based on the definition of $G_n$, for any $i,j$ with $i<j$, the pair $\{a_i, a_j\}$ has no co-neighbor in $\{b_j, b_{j+1}, \ldots, b_{\ceiling{n/2}}\}$
and the pair $\{b_i, b_j\}$ has no co-neighbor in $\{a_j, a_{j+1},\ldots, a_{\floor{n/2}}\}$.
Suppose for a contradiction that $G$ contains a copy $H$ of $R_t$, for some $t$. Suppose $V(R_t)=
\{x_0, y_0, x_1, y_1,$ 
$\dots, x_{t-1}, y_{t-1}\}$ and $E(R_t)=\bigcup_{i=0}^{t-1}\{ x_iy_ix_{i+1}, x_iy_iy_{i+1}\}$.
 For each $v$ in $R_t$, let $v'$ denote its image in $G_n$ under
a fixed isomorphism from $R_t$ to $H$. 
 For any $w$ in $A$ (or $B$), let $\iota(w)$ denote its subscript in $A$ (or $B$). In other words,
if $w=a_\ell$, then $\iota(w)=\ell$.
There are two cases to consider.

\medskip

{\it Case 1.} For some $i\in \{0,\ldots, t-1\}$, $x'_i$ and $y'_i$ are in the same set.

\medskip

Without loss of generality, we may assume that $i=0$ and that $x'_0, y'_0$ are both in $A$.
Then $x'_1, y'_1$ must both be in $B$. Furthermore,
by the observation we made at the beginning of this proof, $\max\{\iota(x'_1), \iota(y'_1)\}<
\max\{\iota(x'_0),  \iota(y'_0)\}$. By repeating this argument, we get 
 $\max\{\iota(x_0),\iota(y_0)\}<\max\{\iota(x_{t-1}),\iota(y_{t-1})\}<\cdots<\max\{\iota(x_0),\iota(y_0)\}$,
which is a contradiction.

\medskip

{\it Case 2.} For all $i\in \{0,\ldots, t-1\}$, $x'_i$ and $y'_i$ are in different sets. 

\medskip

By the symmetry of $R_t$, we  may assume that all the $x'_i$'s are in $A$ and all
the $y'_i$'s are in $B$. Based on the observation we made at the beginning of the proof,
we now must have  $\max\{\iota(x'_i),\iota(y'_i)\}<\max\{\iota(x'_{i+1}),\iota(y'_{i+1})\}$ for all $0\leq i\leq t-1$ (with addition defined modulo $t$). This leads to a contradiction like in Case 1.

\end{proof}

We now prove the main result of this section. This follows immediately from the following lemma.  Given a $3$-graph $G$ and a vertex $x$, the
{\it link graph} $L(x)$ of $x$ is a $2$-graph whose edges are all the pairs $ab$ such that
$xab\in E(G)$.

\begin{lemma} \label{ring-upper}
$\pi(\cR_{\leq t})\leq \frac{1}{2}+\frac{1}{t-1}$.
\end{lemma}

\begin{proof}
By Proposition
\ref{min-degree} and Lemma \ref{familyblowup}, it suffices to prove
$\pi_1(\cR^*_{\leq t})\leq \frac{1}{2}+\frac{1}{t-1}$.
Let $n$ be sufficiently large as a function of $t$.  
Let $G$ be a $3$-graph on $n$ vertices with $\delta_1(G)\geq (\frac{1}{2}+\frac{1}{t-1}) \binom{n-1}{2}\geq (\frac{1}{2}+\frac{1}{t})\binom{n}{2}$. We prove that $G$ contains a member of $\cR^*_{\leq t}$. 
Create an auxiliary digraph with vertex set $\binom{V(G)}{2}$ (all $2$-subsets of $V(G)$) where $(\{u,v\}, \{u',v'\})$ is an edge of $D$ if and only if $uvu'$ and $uvv'$ are edges of $G$ (in other words, if and only if $uv$ is in the link graph of both $u'$ and $v'$).  Let $N=\binom{n}{2}$. 

Let $\{u,v\}$ be a vertex in $D$.  Since $\delta_1(G)\geq (\frac{1}{2}+\frac{1}{t})\binom{n}{2}$, the link graph of $u$ has at least $(\frac{1}{2}+\frac{1}{t})\binom{n}{2}$ edges and the link graph of $v$ has 
at least $(\frac{1}{2}+\frac{1}{t})\binom{n}{2}$ edges.  Therefore there are at least $\frac{2}{t} \binom{n}{2}$ edges in the intersection of their link graphs, which implies $\delta^-(D)\geq \frac{2}{t} N$.  So we can apply Theorem \ref{ChSz} to the directed graph $D$ to obtain a directed cycle $C$ of length at most $\frac{2N}{\frac{2}{t} N+1}\leq t$.  Notice that the directed cycle $C$ corresponds to a subgraph of $G$
which is a member of $\cR^*_{\leq t}$.  
\end{proof}

Proposition \ref{ring-lower} and Lemma \ref{ring-upper} now yield

\begin{theorem}
$\pi(\cR)=\frac{1}{2}$.
\end{theorem}

We have now determined the Tur\'an density of the entire family of rings.
However,  computing its value for any single member $R_t$ appears to be difficult. 
After all, $R_2$ is just $K_4^3$ and determining $\pi(K_4^3)$ has been notoriously difficult.  
A quick observation that one can make is

\begin{proposition}\label{evenupper}
For any positive integers $p,q$ we have $\pi(R_{pq})\leq \pi(R_p)$. Thus, for all even $t$, we have $\pi(R_t)\leq \pi(K_4^3)$.
\end{proposition}

\begin{proof}
Since $R_{pq}$ is contained in the $q$-blowup of $R_p$, we have $\pi(R_{pq})\leq \pi(R_p(q))=\pi(R_p)$.  Now suppose $t$ is even.  Since $R_2=K_4^3$, we have $\pi(R_t)\leq \pi(R_2)=\pi(K_4^3)$.
\end{proof}

Recall that the conjectured value for $\pi(K_4^3)$ is $\frac{5}{9}$.  For the lower bound,
Tur\'an's construction $T(n)$ is obtained by partitioning $n$ vertices as equally as possible into three sets $V_1, V_2, V_3$ and including as edges all triples of the form, $v_1v_2v_3$, $u_1v_1v_2$, $u_2v_2v_3$, $u_3v_3v_1$ for all $u_i, v_i\in V_i$ (see Figure \ref{turanexample}).  It is straightforward to check that if $T(n)$ contains $R_t$ for 
some $t$, then $t$ must be divisible by $3$. Hence, $T(n)$ contains no $R_t$ when
$t\equiv 1,2\pmod{3}$. So we have the following.

\begin{proposition}\label{evenlower}
For $t\equiv 1,2 \mod 3$, $\pi(R_t)\geq \frac{5}{9}$.
\end{proposition}

So by Propositions \ref{evenupper} and \ref{evenlower}, if Tur\'an's conjecture is true, then we would have
$\pi(R_t)=\frac{5}{9}$ for every even $t$ with $t\equiv 1,2\pmod{3}$.

\begin{figure}[ht]
\centering
\scalebox{.85}{\subfloat[Tur\'an's construction $T(n)$]{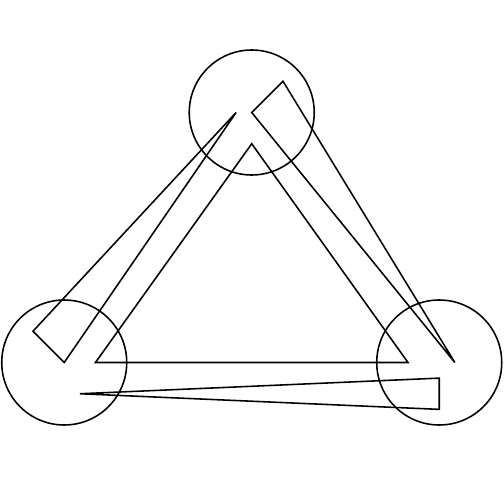
\label{turanexample}}
}~~~~~~~~~~~~~~~~~~~~~~~~~
\scalebox{.85}{\subfloat[The construction $S(n)$]{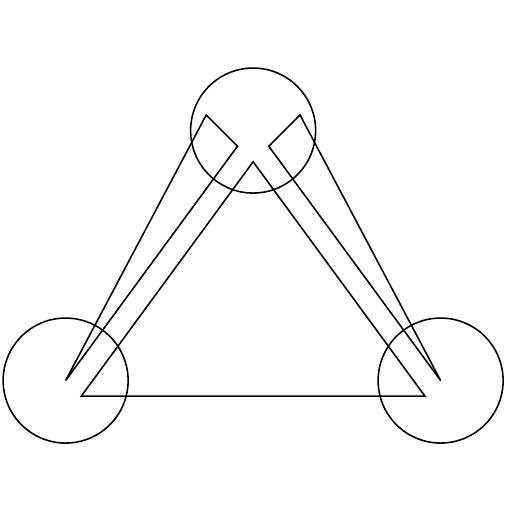
\label{oddexample}}
}
\label{bothexamples}
\caption{}
\end{figure}

Finally, for odd $t$, the following construction shows that 
$\pi(R_t)$ is larger than $\frac{\sqrt{3}}{3}$. Let $S(n)$ be a $3$-graph on $n$ vertices where the vertices are partitioned into three sets $V_1, V_2, V_3$ with sizes $|V_1|=\frac{\sqrt{3}}{3}n$, $|V_2|=|V_3|=(\frac{1}{2}-\frac{\sqrt{3}}{6})n$ whose edges are all triples of the form, $u_1v_1x$, $v_1v_2v_3$ for all $u_i, v_i\in V_i$ and $x\in V_2\cup V_3$ (see Figure \ref{oddexample}).  It is easy to check that $\lim_{n\to \infty} e(S(n))/\binom{n}{3}=\frac{\sqrt{3}}{3}$ and that if $S(n)$ contains $R_t$ then $t$ must be even.  Furthermore, we can iterate this construction inside $V_2$ and $V_3$ to push the density above $\frac{\sqrt{3}}{3}$ while maintaining the fact that there are no odd rings.  

Suppose $|V_1|=(1-\alpha)n$, then the density of $S(n)$ before iterating is $3\alpha(\frac{\alpha^2}{2}-\frac{3\alpha}{2}+1)$; this gives an optimal value of $\frac{\sqrt{3}}{3}> .57735$ when $\alpha=1-\frac{\sqrt{3}}{3}$.  After iterating, the density becomes $3\alpha(\frac{\alpha^2}{2}-\frac{3\alpha}{2}+1)\sum_{i\geq 0}\left(\frac{2}{8^i}\right)^i\alpha^{3i}$; numerical methods give an approximate optimal value of $.588863$ 
when $\alpha=1-\frac{\sqrt{3}}{3}+.015908$.  Thus we have the following.

\begin{proposition}\label{oddlower}
For odd $t$, $\pi(R_t)>.588863$.
\end{proposition}

\begin{definition}
Let $Q_3$ be obtained by adding the edges $x_1y_1x_0, x_1y_1y_0$ to $R_3$.  
\end{definition}

Note the following simple observation.

\begin{proposition}\label{Q3}
For odd $t$ at least $5$, $R_t$ is contained in the blow-up of $Q_3$.
\end{proposition}

The final results in this section are obtained by using Razborov's flag algebra calculus.  Since the upper bounds are not tight (and we don't intend to formally publish the bounds obtained from these calculations), we refer the reader to \cite{R}, \cite{R2}, \cite{BT} for an explanation of the method and its applications.

\begin{proposition}\label{uppers}~
\begin{enumerate}
\item 
For even $t$, $\pi(R_t)<.561666$.

\item For odd $t\geq 5$, $\pi(R_t)< 0.594312$

\item For $t$ an odd multiple of $3$, $\pi(R_t)< .594258$.

\item For $t$ an even multiple of $3$, $\pi(R_t)< .512303$.
\end{enumerate}
\end{proposition}

\begin{proof}
In each case we use Lemma \ref{codegreeblowup} to transfer a statement about the blow-up of a graph to a statement about $R_t$.
\begin{enumerate}
\item Proposition \ref{evenupper} shows $R_t$ is contained in the blow-up of $K_4^3$ and flag algebra calculations give $\pi(K_4^3)<.561666$ (see \cite{R}).

\item Proposition \ref{Q3} shows $R_t$ is contained in the blow-up of $Q_3$ and flag algebra calculations give $\pi(Q_3)< 0.594312$.

\item $R_t$ is contained in the blow-up of $R_3$ and flag algebra calculations give $\pi(R_3)< .594258$

\item $R_t$ is contained in the blow-up of $R_2$ and $R_3$ and flag algebra calculations give $\pi(\{R_2, R_3\})< .512303$

\end{enumerate}
\end{proof}

The results of this section are summarized below, with the lower bounds coming from Propositions \ref{ring-lower}, \ref{evenlower}, \ref{oddlower} and the upper bounds coming from Proposition \ref{uppers}.  Perhaps the most interesting thing about the upper bounds for rings is that it is possible to get nearly tight results for every value of $t$ using only flag algebra calculations for small $3$-graphs and Lemma \ref{codegreeblowup}.

\begin{equation*}
\left.
\begin{IEEEeqnarraybox}[
\IEEEeqnarraystrutmode
\IEEEeqnarraystrutsizeadd{2pt}
{2pt}
][c]{lCl}
\text{ if } t\equiv 0\bmod 6, ~~~1/2\\
\text{ if } t\equiv 1\bmod 6, ~~~.588863\\
\text{ if } t\equiv 2\bmod 6, ~~~5/9\\
\text{ if } t\equiv 3\bmod 6, ~~~.588863\\
\text{ if } t\equiv 4\bmod 6, ~~~5/9\\
\text{ if } t\equiv 5\bmod 6, ~~~.588863
\end{IEEEeqnarraybox}
\, \right\} \quad
\leq\pi(R_t)< \quad
\left\{ \,
\begin{IEEEeqnarraybox}[
\IEEEeqnarraystrutmode
\IEEEeqnarraystrutsizeadd{2pt}
{2pt}
][c]{lCl}
.512303\\
.594312\\
.561666\\
.594258\\
.561666\\
.594312
\end{IEEEeqnarraybox}
\right.
\label{ringsummary}
\end{equation*}

%
%

Given the results of this section and Theorem \ref{ring-codegree}, it would be interesting to solve the following problem.

\begin{problem}
Determine $\pi(R_t)$ or $\pi_2(R_t)$ for each fixed value of $t$.
\end{problem}

\section{Concluding remarks} \label{remarks}

Let $q$ be a prime power and let $PG_2(q)$ be $(q+1)$-graph with vertex set equal to the one dimensional subspaces of $\mathbb{F}_q^3$ and edges corresponding to the two-dimensional subspaces of $\mathbb{F}_q^3$.  We call $PG_2(q)$ the \emph{projective geometry} of dimension $2$ over $\mathbb{F}_q$; note that $PG_2(2)$ is the Fano plane.  
In \cite{K}, Keevash also proved the following more general theorem about projective geometries
\begin{theorem}\label{KeevashPG}
$\ex_q(PG_2(q))\leq \frac{n}{2}$
\end{theorem}
\noindent
Furthermore, there is a nearly matching lower bound when $q$ is an odd prime power (see \cite{KZ} and \cite{K}).  The proof we present in Section \ref{fano} relies on the fact that there is a family of $3$-graphs $\mathcal{R}_{\leq t}$, such that each member $R_i\in \mathcal{R}_{\leq t}$ has the $(2i, i+1)$-property and $\pi_2(\mathcal{R}_{\leq t})<\frac{1}{2}$.  In fact, our same proof could be used to give a simple proof of Theorem \ref{KeevashPG} if there was an affirmative answer to the following question.

\begin{problem}\label{lm}
Let $k\geq 4$.  Does there exist a finite family $\mathcal{F}^k$ of $k$-graphs such that for each $F\in \mathcal{F}^k$ there exists a positive integer $t$ such that $F$ has the $(2t, t+1)$-property and $\pi_{k-1}(\mathcal{F}^k)<\frac{1}{2}$.
\end{problem}
\noindent
It seems conceivable that obtaining a $k$-graph with the $(l,m)$ property for some other values of $l$ and $m$ might give us the same benefit and be easier to obtain; however, this is not the case. On one hand, we must have $l\geq 2(m-1)$ so that equation \eqref{exact} holds.  On the other hand, when $m\leq \ceiling{\frac{l}{2}}$ the complete balanced bipartite $k$-graph $B^k(n)$, which has $\delta_{k-1}(B^k(n))\geq \floor{n/2}$, does not contain any subgraph with the $(l,m)$-property (any subgraph of $B^k(n)$ with $l$ vertices must contain an independent set of size $\ceiling{\frac{l}{2}}$).  So we must have $\ceiling{\frac{l}{2}}<m\leq \frac{l}{2}+1$, which implies that $l$ is even and $m=\frac{l}{2}+1$.

A different problem is the following. Instead of determining the co-degree threshold for a single copy of $H$ in $G$, one can ask about the co-degree threshold for $\frac{|V(G)|}{|V(H)|}$ vertex disjoint copies of $H$ in $G$ (assuming $|V(H)|$ divides $|V(G)|$).  This has been referred to as the \emph{tiling} or \emph{factoring} problem and received much attention lately.  Interestingly, the co-degree threshold for tiling with $K_4^3$ and $K_4^3-e$ have been determined (see \cite{KM}, \cite{LM2}), but the co-degree threshold for a single copy of $K_4^3$ or $K_4^3-e$ is still unknown and appears to be difficult.  Since the co-degree threshold for a single copy of $\bF$ is known and seems to be much easier than $K_4^3$ or $K_4^3-e$, it would be interesting to determine the co-degree threshold for tiling with $\bF$.

\begin{problem}
Let $n$ be divisible by $7$ and let $G$ be a $3$-graph on $n$ vertices.  Determine the minimum value $d$ such that $\delta_2(G)\geq d$ implies that $G$ contains $\frac{n}{7}$ vertex disjoint copies of $\bF$.
\end{problem}

The relationship between the edge density of a hypergraph and its subgraphs with large co-degree is  also very intriguing. Even the following simple questions do not seem to have an easy answer. A $k$-graph $H$ is said to {\it cover pairs} if $H$ has at least $k+1$ vertices and every pair of vertices lies in some edge, i.e. $\delta_2(H)\geq 1$.

\begin{problem} \label{problem1}
What is 
$$\limsup_{n\to \infty} \left\{\frac{e(G)}{\binom{n}{k}}: G \subseteq \binom{[n]}{k},  \mbox{ $G$ contains no subgraph that covers pairs} \right\}?$$
\end{problem}

Since $K^3_4-e$ covers pairs, for $k=3$ the answer to Problem \ref{problem1} is
certainly no more than $\pi(K^3_4-e)$, which is known to be at most $0.2871$. 
\begin{problem}
Given any positive integer $s$, what is 
$$\limsup_{n\to \infty} \left\{\frac{e(G)}{\binom{n}{k}}:  G\subseteq \binom{[n]}{k}, \mbox{ 
$G$ contains no subgraph on $s$ vertices that covers pairs} \right\}?$$
\end{problem}

More generally, one may ask

\begin{problem}
Given positive integers $s,t$, what is 
$$\limsup_{n\to \infty} \left\{\frac{e(G)}{\binom{n}{k}}: G\subseteq\binom{[n]}{k},  \mbox{ 
$G$ contains no subgraph $H$ on $s$ vertices with $\delta_2(H)\geq t$} \right\}?$$
\end{problem}


\end{document}